\documentclass[noend,a4paper]{amsproc}

\usepackage{amsthm,amsmath,amsfonts,mathrsfs,amssymb}
\usepackage{algorithm}
\usepackage{algorithmicx, algpseudocode}
\usepackage[T1]{fontenc}   
\usepackage{multirow}
\usepackage{natbib}
\usepackage{breakurl}
\usepackage{hyperref}

\newtheorem{theorem}{Theorem}
\newtheorem{defn}[theorem]{Definition}
\newtheorem{prop}[theorem]{Proposition}
\newtheorem{lemma}[theorem]{Lemma}
\theoremstyle{definition}
\newtheorem{remark}[theorem]{Remark}

\newcommand{\algorithmicbreak}{\textbf{break}}
\newcommand{\Break}{\State \algorithmicbreak}

\newcommand{\Singular}{\textsc{Singular}}
\newcommand{\realclassify}{\texttt{realclassify.lib}}
\newcommand{\classify}{\texttt{classify.lib}}

\newcommand{\requiv}{\ensuremath{\mathrel{\overset{r}{\sim}}}}
\newcommand{\cequiv}{\ensuremath{\mathrel{\overset{c}{\sim}}}}
\newcommand{\sequiv}{\ensuremath{\mathrel{\overset{s}{\sim}}}}

\DeclareMathOperator{\ord}{ord}
\DeclareMathOperator{\m}{\mathfrak{m}}
\DeclareMathOperator{\jet}{jet}
\DeclareMathOperator{\corank}{corank}

\DeclareMathOperator{\diag}{diag}
\DeclareMathOperator{\NF}{NF}
\DeclareMathOperator{\N}{\mathbb{N}}
\DeclareMathOperator{\Q}{\mathbb{Q}}
\DeclareMathOperator{\R}{\mathbb{R}}
\DeclareMathOperator{\C}{\mathbb{C}}
\DeclareMathOperator{\K}{\mathbb{K}}
\DeclareMathOperator{\A}{\mathbb{A}}

\DeclareMathOperator{\boldzero}{\mathbf{0}}

\title[The Classification of Real Singularities Using \textsc{Singular}, %
Part I]%
{The Classification of Real Singularities Using \textsc{Singular}\\
Part I: Splitting Lemma and Simple Singularities}

\author{Magdaleen S. Marais}
\address{Magdaleen S. Marais\\
African Institute for Mathematical Sciences and University of Pretoria\\
Department of Mathematics and Applied Mathematics\\
0002 Pretoria\\
South Africa}
\email{magdaleen@aims.ac.za}

\author{Andreas Steenpa\ss}
\address{Andreas Steenpa\ss\\
Department of Mathematics\\
University of Kaiserslautern\\
Erwin-Schr\"odinger-Str.\\
67663 Kaiserslautern\\
Germany}
\email{steenpass@mathematik.uni-kl.de}

\thanks{This research was supported by the African Institute for Mathematical
Sciences and a grant awarded by Gert-Martin Greuel. We are thankful to both of
them.}


\begin{document}

\begin{abstract}
We present algorithms to classify isolated hypersurface singularities over the
real numbers according to the classification by V.I.~Arnold \citep{AVG1985}.
This first part covers the splitting lemma and the simple singularities; a
second and a third part will be devoted to the unimodal singularities up to
corank~2. All algorithms are implemented in the \Singular{} library
\realclassify{} \citep{realclassify}.
\end{abstract}

\maketitle

\section{Introduction}
\citet{AVG1985} present classification theorems for singularities over the
complex numbers up to modality~2 and for singularities over the real numbers up
to modality~1, including complete sets of normal forms. For the complex case,
they also give an algorithm how the type of a given singularity can be
computed, called the ``determinator of singularities''
\citep[cf.\@][ch.~16]{AVG1985}, but this question is left open for the real
case. The goal of this paper, together with its subsequent parts, is to fill
this gap. For this purpose, we present both, algorithms and an implementation
thereof, for the classification of isolated hypersurface singularities up to
modality~1 and corank 2 over the real numbers w.r.t.\@ right equivalence.

We consider real functions with a critical point at the origin and critical
value~$0$, i.e.\@ functions in $\m^2$, where $\m$ denotes the ideal of function
germs vanishing at the origin. Two function germs $f, g \in \m^2 \subset
\R[[x_1,\ldots,x_n]]$ are considered as right equivalent, denoted by
$f \requiv g$, if there exists an $\R$-algebra automorphism $\phi$ of
$\R[[x_1,\ldots,x_n]]$ such that $\phi(f) = g$.

We have implemented all the algorithms presented here in the computer algebra
system \Singular{} \citep{DGPS}. The implementation is freely available as a
\Singular{} library called \realclassify{} which relies on \Singular's
\classify{} to determine, for a given polynomial, the type in Arnold's
classification over the complex numbers. The methods used in \classify{} will
not be discussed in this paper. For more information in this regard,
\citet{Kruger} can be studied.

In Section~\ref{sec:prerequisites}, we introduce basic notions and methods
which are frequently used for the algorithmic classification in the subsequent
sections. We first give an overview of the different notions of equivalence in
singularity theory and how they are related in
Subsection~\ref{subsec:equivalence}. Thereafter we recall some basic results on
the Milnor number and the determinacy in Subsections~\ref{subsec:milnor_number}
and~\ref{subsec:determinacy}, and we also recall how these invariants can be
computed. As a further prerequisite, we show that the homogeneous parts of
lowest degree of two right equivalent functions factorize in the same way over
$\R$ (Section~\ref{subsec:factorization}, Proposition~\ref{kjet}). We also show
that in some cases, this factorization can even be carried out over $\Q$ which
is important for the algorithmic aspect (Lemma~\ref{x^3}).

Using the Splitting Lemma (Theorem~\ref{thm:splitting_lemma}), any function
germ $f$ over the real numbers with an isolated singularity at the origin can
be written, after choosing a suitable coordinate system, as the sum of two
functions of which the variables are disjoint. One of the functions, called the
nondegenerate part of $f$, is a nondegenerate quadratic form and the other
function, called the residual part of $f$, is an element of $\m^3$. The number
of variables in the residual part is equal to the corank of $f$, denoted by
$\corank(f)$. In this paper we only consider germs with corank $0$, $1$
and~$2$. A version of the Splitting Lemma for singularities over $\R$ and a
corresponding algorithm are discussed in Section~\ref{sec:splitting_lemma}.

In \citet{AVG1985}, the real singularities of modality $0$ and $1$ are
classified up to stable equivalence into main types which split up into more
subtypes
depending on the sign of certain terms. Two functions are stably equivalent if
they are right equivalent after the direct addition of nondegenerate quadratic
forms. Hence after applying the Splitting Lemma, we only need to consider the
residual part in order to compute the correct subtype. It can be easily seen
that the subtypes are complex equivalent to a complex singularity type of the
same name as its corresponding real main singularity type (see
Table~\ref{tab:normal_forms}). In fact there is a bijection between the complex
types of modality $0$ and $1$ and the real main types. Thus, if we can
determine the complex type of a function germ, we only need to determine the
correct subtype of the corresponding real main type. The classification of the
residual part is given in Section~\ref{sec:residual_part}, together with
explicit algorithms for each singularity type.

\section{Prerequisites}%
\label{sec:prerequisites}

\subsection{Equivalence}%
\label{subsec:equivalence}

There are different notions for the equivalence of two power series in
singularity theory:

\begin{defn}%
\label{def:equivalence}
Let $\K$ be either $\R$ or $\C$ and let $f, g \in \K[[x_1, \ldots, x_n]]$ be
two power series.

\begin{enumerate}
\item
$f$ and $g$ are called \emph{right equivalent}, denoted by $f \requiv g$, if
there exists a $\K$-algebra automorphism $\phi$ of $\K[[x_1, \ldots, x_n]]$ such
that
\[
\phi(f) = g \,.
\]

\item
$f$ and $g$ are called \emph{contact equivalent}, denoted by $f \cequiv g$, if
there exist a $\K$-algebra automorphism $\phi$ of $\K[[x_1, \ldots, x_n]]$ and a
unit $u \in \K[[x_1, \ldots, x_n]]^*$ such that
\[
\phi(f) = u \cdot g \,.
\]

\item\label{enum:stable_equivalence}
$f$ and $g$ are called \emph{stably equivalent}, denoted by $f \sequiv g$, if
there exist indices $k, l \in \{1, \ldots ,n\}$ such that
$f \in \K[[x_1, \ldots, x_k]]$, $g \in \K[[x_1, \ldots, x_l]]$, and the two
power series become right equivalent after the addition of nondegenerate
quadratic forms in the additional variables, i.e.
\begin{alignat*}{4}
& f(x_1, \ldots, x_k) &{}\pm{}& x_{k+1}^2 &{}\pm{}& \ldots &{}\pm{}& x_n^2 \\
\requiv{} \,
& g(x_1, \ldots, x_l) &{}\pm{}& x_{l+1}^2 &{}\pm{}& \ldots &{}\pm{}& x_n^2 \,.
\end{alignat*}
\end{enumerate}
\end{defn}

\begin{remark}
Note that right equivalence implies both contact and stable equivalence, but
the converse statements are not true in general. For instance, $x_1^2+x_2^2$
and $-x_1^2-x_2^2$ are contact, but not right equivalent over $\R$.
\end{remark}

This article and the \Singular{} library \realclassify{} both deal with the
classification of the simple singularities w.r.t.\@ \emph{right} equivalence
over $\K = \R$. We first use the Splitting Lemma and
Algorithm~\ref{alg:splitting_lemma} from Section~\ref{sec:splitting_lemma} to
get rid of the nondegenerate part. We can then apply the classification by
\citet{AVG1985} w.r.t.\@ stable equivalence to the residual part in
Section~\ref{sec:residual_part}.

From the point of view of real algebraic geometry, a classification w.r.t.\@
contact rather than right equivalence might be more interesting because it
better reflects the local real geometry of a singularity. In the example from
the remark above, $x_1^2+x_2^2$ and $-x_1^2-x_2^2$ both define a solitary point
in the plane, as opposed to the two intersecting lines defined by
$-x_1^2+x_2^2$ and $x_1^2-x_2^2$. But note that a classification w.r.t.\@ right
equivalence is only finer than one based on contact equivalence. Hence the
shape of the local real geometry of a singularity can always be read off from
its right equivalence class, given by its stable equivalence class together
with the inertia index introduced in Theorem~\ref{thm:splitting_lemma}; the
\Singular{} library \realclassify{} indeed also serves this purpose. For the
simple singularities, it is moreover easy to see which of the right equivalence
classes are contact equivalent.

\subsection{The Milnor Number}%
\label{subsec:milnor_number}

We briefly recall the following well-known definition:

\begin{defn}
For $f \in \R[[x_1,\ldots,x_n]]$ and $p \in \A_{\R}^n$, the
\emph{Milnor number} of $f$ at $p$ is defined as
\[
\mu(f, p) := \dim_{\R}
\left( \R[[x_1-p_1, \ldots, x_n-p_n]] \bigg/
\left\langle \frac{\partial f}{\partial x_1}, \ldots,
\frac{\partial f}{\partial x_n} \right\rangle \right)
\in \N \cup \{\infty\} \,.
\]
If $p$ is the origin, we simply write $\mu(f)$ instead of $\mu(f, p)$.
\end{defn}

The Milnor number is known to be finite at isolated singularities
\citep[cf.\@][Chapter~I, Lemma~2.3]{GLS2007} and to be invariant under right
equivalence
(cf.\@ Lemma~2.10 ibid.). It is thus an important tool for the classification
of isolated singularities. We refer to \citet{GLS2007} for more properties of
this invariant.

There is a well-known algorithm for the computation of the Milnor number which
is implemented in \Singular{}, see \citet{GP2008}, pp. 526-528.

\subsection{The Determinacy}%
\label{subsec:determinacy}

In general, the singularities we deal with in this paper are defined by power
series, but algorithmically, we want to work with polynomials. It is thus
important for our algorithmic approach that any power series defining an
isolated singularity is right equivalent to a polynomial which can be obtained
from it by leaving out terms of sufficiently high order.

\begin{defn}
Let $f \in \R[[x_1,\ldots,x_n]]$ be a power series.

\begin{enumerate}
\item Let $f = \sum_{j=0}^\infty f_j$ be the decomposition of $f$ into
homogeneous parts $f_j$ of degree $j$.
For $k \in \N$, we define the \emph{$k$-jet} of $f$ as
\[
\jet(f,k) := \sum_{i=0}^k f_i \,.
\]
In other words, the $k$-jet of $f$ can be obtained from $f$ by leaving out all
terms of order higher than $k$.

\item $f$ is called \emph{$k$-determined} if
\[
\forall g \in \m^{k+1}: \quad f \requiv \jet(f,k)+g \,.
\]
\end{enumerate}
\end{defn}

The determinacy is, just as the Milnor number, both invariant under right
equivalence and finite for isolated singularities. We cite the following
statement \citep[cf.\@][Chapter~I, Supplement to Theorem~2.23]{GLS2007} due to
its importance for the algorithmic approach and refer to \citet{GLS2007} for
further results regarding the determinacy:

\begin{prop}\label{prop_determinacy}
Let $f \in \m \subset \R[[x_1,\ldots,x_n]]$. If
\[
\m^{k+1} \subset \m^2 \left\langle \frac{\partial f}{\partial x_1}, \ldots,
\frac{\partial f}{\partial x_n} \right\rangle_{\R[[x_1,\ldots,x_n]]}
\]
holds, then $f$ is $k$-determined.
\end{prop}

As a consequence of this, any power series $f$ which has an isolated
singularity at the origin is $(\mu(f)+1)$-determined
\citep[cf.\@][Chapter~I, Corollary~2.24]{GLS2007}.
But we can often compute a much better upper bound for the
determinacy by using the above statement as in Algorithm~\ref{alg_Determinacy}.

\begin{algorithm}[ht]
\caption{\label{alg_Determinacy}Determinacy}
\begin{algorithmic}[1]

\Require{$f \in \Q[x_1,\ldots,x_n]$ with an isolated singularity at the origin}
\Ensure{an upper bound for the determinacy of $f$}

\State $k := \textsc{Milnor}(f)+1$
\State $J := \left( \frac{\partial f}{\partial x_1}, \ldots,
\frac{\partial f}{\partial x_n} \right) \subset \Q[x_1,\ldots,x_n]$
\State compute a standard basis $G$ of $(\m^2 J)$ w.r.t.\@ a local monomial
ordering $<$
\For{$(l = 1,\ldots,k-1)$}
\If{$(\NF_<(\m^{l+1},G) = 0)$}
\State $k := l$
\Break
\EndIf
\EndFor
\Return $k$

\end{algorithmic}
\end{algorithm}

\begin{remark}
In Algorithm~\ref{alg_Determinacy}, the for-loop computes the minimal
$k \in \N$ such that the condition in Proposition~\ref{prop_determinacy} holds.
This number is equal to the degree of the so-called highest corner
of $\langle G \rangle = (\m^2 J)$ \citep[cf.\@][Corollary~A.9.7]{GP2008}
and can thus also be computed by combinatorial means with the
\Singular{} command \verb+highcorner()+ which is often much faster.
\end{remark}

It is worth to note that the Milnor number of an arbitrary power series
$f \in \R[[x_1,\ldots,x_n]]$ and the determinacy of a semi-quasihomogeneous
power series $f \in \R[[x_1,\ldots,x_n]]$ do not change if we regard $f$ as an
element of $\C[[x_1,\ldots,x_n]]$. The same holds for the output of the
corresponding algorithms presented here.

\subsection{Results Regarding the Factorization of Homogeneous Polynomials
over $\R$ and $\Q$}%
\label{subsec:factorization}

\begin{defn}
Let $\phi$ be an $\R$-algebra automorphism of $\R[[x_1,\ldots,x_n]]$. For
$j \geq 0$ we define the \emph{$j$-jet} of $\phi$, denoted by $\phi_j$, to be
the automorphism given by
\[
\phi_j(x_i) := \jet(\phi(x_i),\, j+1) \quad \forall i = 1,\ldots,n \,.
\]
\end{defn}

The next result is in many cases a starting point for the algorithmic
classification of the residual part, see Section~\ref{sec:residual_part}. Given
$f$ and $g$ with $f \requiv g$, it can be used to determine $\phi_0$ for some
automorphism $\phi$ such that $\phi(f) = g$.

\begin{prop}\label{kjet}
Let $f,g \in \R[[x_1,\ldots,x_n]]$ be two power series with $f \requiv g$ and
$k := \ord(f) > 1$. Let $\phi$ be an $\R$-algebra automorphism of
$\R[[x_1,\ldots,x_n]]$ such that $\phi(f)=g$.

If $\jet(f,k)$ factorizes as
\[
\jet(f,k) = f_1^{s_1} \cdots f_t^{s_t}
\]
in $\R[x_1,\ldots,x_n]$, then $\jet(g,k)$ factorizes as
\[
\jet(g,k) = \phi_0(f_1)^{s_1} \cdots \phi_0(f_t)^{s_t} \,.
\]
\end{prop}

\begin{proof}
By assumption we have that $f = f_1^{s_1} \cdots f_t^{s_t} + f'$, where
$f_1^{s_1} \cdots f_t^{s_t}$ is homogeneous of degree $k$ and the order of $f'$
is greater than $k$. We denote the higher order parts of $\phi$ by
$\phi^* := \phi-\phi_0$. Since $\phi$ is a homomorphism, it follows that
\begin{align*}
\phi(f) &= \phi(f_1^{s_1} \cdots f_t^{s_t}) + \phi(f') \\
&= \phi_0(f_1^{s_1} \cdots f_t^{s_t})
+ \phi^*(f_1^{s_1} \cdots f_t^{s_t}) + \phi(f')
\end{align*}
where $\phi_0(f_1^{s_1} \cdots f_t^{s_t})$ is homogeneous of degree $k$ and
both $\phi^*(f_1^{s_1} \cdots f_t^{s_t})$ and $\phi(f')$ are of order higher
than $k$. Hence
\[
\jet(g, k) = \jet(\phi(f), k) = \phi_0(f_1^{s_1} \cdots f_t^{s_t})
= \phi_0(f_1)^{s_1} \cdots \phi_0(f_t)^{s_t} \,.
\]
\end{proof}

Since we do not want to work with rounding errors nor field extensions in the
implementation of the proposed algorithms, the above result would not be of
much help for this purpose without the following result.

\begin{lemma}\label{x^3}
If $f \in \Q[x,y]$ is homogeneous and factorizes as
\[
\text{(i) } g_1^d \text{ or (ii) } g_1 g_2^d \,,
\]
where $g_1, g_2 \in \R[x,y]$ are polynomials of degree $1$ and $d > 1$, then
$f$ factorizes as
\[
\text{(i) } ag_1'^d \text{ or (ii) } ag_1' g_2'^d \,,
\]
respectively, where $g_1', g_2' \in \Q[x,y]$ are polynomials of degree $1$ and
$a \in \Q$.
\end{lemma}

\begin{proof}

(i) Let $f = (a_1x+a_2y)^d$, $a_1, a_2 \in \R$. Without loss of generality,
suppose $a_1 \neq 0$. Then $f = a_1^d(x+\frac{a_2}{a_1}y)^d$. Since the
coefficient of $x^d$ in $f \in \Q[x,y]$ is $a_1^d$, we have $a_1^d \in \Q$ and
therefore $(x+\frac{a_2}{a_1}y)^d \in \Q[x,y]$ which, by dehomogenization,
leads to $(x+\frac{a_2}{a_1})^d \in \Q[x]$. Since $\Q$ is a perfect field it
follows that $\frac{a_2}{a_1} \in \Q$. Thus $f = ag_1'^d$, where
$a := a_1^d \in \Q$ and $g_1' = x+\frac{a_2}{a_1}y \in \Q[x,y]$.

(ii) Let $f = (a_1x+a_2y)(a_3x+a_4y)^d$, $a_1,\ldots,a_4 \in \R$. Suppose
$a_1,a_3 \neq 0$. For the cases $a_1,a_4 \neq 0$, $a_2,a_3 \neq 0$ and
$a_2,a_4 \neq 0$ the proofs are similar. We have $a_1a_3^d \in \Q$ analogously
to part (i). Hence $(x+\frac{a_2}{a_1}y)(x+\frac{a_4}{a_3}y)^d \in \Q[x,y]$
which in turn implies $(x+\frac{a_2}{a_1})(x+\frac{a_4}{a_3})^d \in \Q[x]$.
Since $\Q$ is a perfect field it follows that the roots of this polynomial are
rational. Therefore $f = ag_1'g_2'^d$ with $a: = a_1a_3^d \in \Q$,
$g_1' := (x+\frac{a_2}{a_1}y) \in \Q[x,y]$, and
$g_2' := (x+\frac{a_4}{a_3}y) \in \Q[x,y]$.
\end{proof}

\section{The Splitting Lemma}%
\label{sec:splitting_lemma}

\begin{defn}
For $f \in \R[[x_1,\ldots,x_n]]$, we define the corank of $f$, denoted by
$\corank(f)$, as the corank of the Hessian matrix $H(f)$ at $\boldzero$, i.e.
\[
\corank(f) := \corank(H(f)(\boldzero)) \,.
\]
\end{defn}

The following well-known theorem, called the Splitting Lemma, allows us to
reduce the classification to germs of full corank or, algorithmically, to
a polynomial contained in $\m^3 \cap \R[x_1,\ldots,x_c]$ for a given input
polynomial of corank $c$. We present a version for singularities over the real
numbers, taking into account the signs of the squares.

\begin{theorem}\label{thm:splitting_lemma}
If $f \in \m^2 \subset \R[[x_1,\ldots,x_n]]$ has an isolated singularity and if
its corank is $c$, then
\[
f \requiv g -\sum_{i=c+1}^{c+\lambda} x_i^2 +\sum_{i=c+\lambda+1}^n x_i^2
\]
with $g \in \m^3 \cap \R[[x_1,\ldots,x_c]]$. $g$ is called the residual part of
$f$ and $\lambda$ is called the inertia index of~$f$. Both $\lambda$ and the
right equivalence class of $g$ are uniquely determined by $f$.
\end{theorem}

The following proof is based upon the proofs of Theorems 2.46 and 2.47 in
Chapter~I of \citet{GLS2007}.

\begin{proof}
The corank of the Hessian matrix of $f$ at $0$ is $c$, so by the theory of
quadratic forms over $\R$ there is a transformation matrix $T$ such that
\[
T^t \cdot {\textstyle\frac{1}{2}} H(f)(\boldzero) \cdot T
= \diag(0,\ldots,0,-1,\ldots,-1,1,\ldots,1) \,.
\]
Therefore the linear coordinate change
$(x_1,\ldots,x_n) \mapsto (x_1,\ldots,x_n) \cdot T^t$ transforms the 2-jet of
$f$ into
$\left(-\sum_{i=c+1}^{c+\lambda} x_i^2 +\sum_{i=c+\lambda+1}^n x_i^2\right)$
where $\lambda$ is the inertia index of $f$.
Applied to $f$, this transformation leads to
\begin{align*}
f^{(3)} (x_1,\ldots,x_n)
  :\!&= f((x_1,\ldots,x_n) \cdot T^t) \\
  &= g_3
  -\sum_{i=c+1}^{c+\lambda} x_i^2 +\sum_{i=c+\lambda+1}^n x_i^2
  +\sum_{i=c+1}^n x_i\cdot h_i^{(3)}
\end{align*}
with $g_3 \in \m^3 \cap \R[[x_1,\ldots,x_c]]$ and $h_i^{(3)} \in \m^2$. The
coordinate change $\phi^{(3)}$ defined by
\[
\phi^{(3)}(x_i) :=
\begin{cases}
x_i,                      &i = 1, \ldots, c, \\
x_i+\frac{1}{2}h_i^{(3)}, &i = c+1, \ldots, c+\lambda, \\
x_i-\frac{1}{2}h_i^{(3)}, &i = c+\lambda+1, \ldots, n,
\end{cases}
\]
yields
\begin{align*}
f^{(4)} (x_1,\ldots,x_n)
  :\!&= f^{(3)}(\phi^{(3)}(x_1,\ldots,x_n)) \\
  &= g_3 +g_4
  -\sum_{i=c+1}^{c+\lambda} x_i^2 +\sum_{i=c+\lambda+1}^n x_i^2
  +\sum_{i=c+1}^n x_i\cdot h_i^{(4)}
\end{align*}
with $g_4 \in \m^4 \cap \R[[x_1,\ldots,x_c]]$ and $h_i^{(4)} \in \m^3$.
Continuing in the same manner, the last sum will be of arbitrarily high order.
It can be eventually left out because $f$ is finitely determined as an isolated
singularity.
\end{proof}

Since this proof is constructive, we can immediately derive Algorithm
\ref{alg:splitting_lemma} from it.

\begin{algorithm}[ht]
\caption{Algorithm for the Splitting Lemma}%
\label{alg:splitting_lemma}
\begin{algorithmic}[1]

\Require{$f \in \m^2 \subset \Q[x_1,\ldots,x_n]$ and $k \in \N$ such that $f$
is $k$-determined}

\Ensure{the corank $c$ of $f$, the inertia index $\lambda$ of $f$ and
$g \in \m^3 \cap \Q[x_1,\ldots,x_c]$ such that
\[
f \requiv g -\sum_{i=c+1}^{c+\lambda} x_i^2 +\sum_{i=c+\lambda+1}^n x_i^2
\]
}

\State compute a transformation matrix $T \in \R^{n \times n}$ such that
\[
T^t \cdot {\textstyle\frac{1}{2}} H(f)(\boldzero) \cdot T
= \diag(0,\ldots,0,-1,\ldots,-1,1,\ldots,1)
=: N
\]
\State $c :=$ number of zeroes on the diagonal of $N$
\State $\lambda :=$ number of entries equal to $-1$ on the diagonal of $N$
\State $f^{(3)} (x_1,\ldots,x_n) := f((x_1,\ldots,x_n) \cdot T^t)$
\For{$(l = 3, \ldots, k)$}
\State write $f^{(l)}$ as
\[
f^{(l)} = \sum_{j=3}^l g_j
  -\sum_{i=c+1}^{c+\lambda} x_i^2 +\sum_{i=c+\lambda+1}^n x_i^2
  +\sum_{i=c+1}^n x_i\cdot h_i^{(l)}
\]
\hspace{\algorithmicindent}with $g_j \in \m^j \cap \Q[x_1,\ldots,x_c]$ and
$h_i^{(l)} \in \m^{l-1}$
\State $f^{(l+1)} := \phi^{(l)}(f^{(l)})$ where $\phi^{(l)}$ is defined by
\[
\phi^{(l)}(x_i) :=
\begin{cases}
x_i,                      &i = 1, \ldots, c, \\
x_i+\frac{1}{2}h_i^{(l)}, &i = c+1, \ldots, c+\lambda, \\
x_i-\frac{1}{2}h_i^{(l)}, &i = c+\lambda+1, \ldots, n.
\end{cases}
\]
\EndFor

\State $g := \sum_{j=3}^k g_j$
\Return $c, \lambda, g$

\end{algorithmic}
\end{algorithm}

\section{The Real Classification of the Residual Part
\texorpdfstring{w.r.t.\@}{w.r.t.} Stable Equivalence}%
\label{sec:residual_part}

\citet{AVG1985} present independent classifications of the simple singularities
over the complex and over the real numbers, using stable equivalence. We refer
to the equivalence classes of the complex classification as \emph{complex
types}. In the classification over the real numbers, the simple singularities
are divided into \emph{main types} which split up into one or more
\emph{subtypes}. These subtypes differ from each other only in the sign of
certain terms.

It is known that the modality does not decrease under complexification
\citep[pp.~273-274]{AVG1985}. So by applying the algorithms for the complex
classification to the real normal forms, it is easy to see that in
modality~$0$, there is a one-to-one correspondence between the complex types
and the real main types. The real classification can thus be seen as a
refinement of the complex one. As we will see in the subsequent parts of this
series of articles, the same holds true also in modality~$1$, but in both
cases, this is not clear a priori and can only be deduced from the
independently derived complex and real classifications. In fact, it is not
known whether the modality is preserved under complexification in general
\citep[pp.~273-274]{AVG1985}.

Both the real and complex normal forms of the simple singularities are listed
in Table~\ref{tab:normal_forms}. From here onwards we will work with stable
equivalence, cf.\@
Definition~\ref{def:equivalence}(\ref{enum:stable_equivalence}). For all
degenerate forms it is thus only necessary, after applying the Splitting Lemma,
to consider their residual parts, i.e.\@ germs in $\m^3$. Note that the right
equivalence class of a real singularity is given by its stable equivalence
class together with its inertia index which can be computed using
Algorithm~\ref{alg:splitting_lemma}.

\begin{table}[!htb]
\centering
\caption{Real normal forms of singularities of modality $0$.}
\label{tab:normal_forms}
\begin{tabular}{|c|c|c|c|c|}
\hline
& Complex & Normal forms & \multirow{2}{*}{Equivalences} &
\multirow{2}{*}{Values of $k$} \\
& normal form & of real subtypes & & \\
\hline\hline
\multirow{2}{*}{$A_k$} & \multirow{2}{*}{$x^{k+1}$} & $+x^{k+1}$ $(A_k^+)$ &
$A_k^+ \requiv A_k^-$ & \multirow{2}{*}{$k \geq 1$} \\ \cline{3-3}
& & $-x^{k+1}$ $(A_k^-)$ & for~even $k$ & \\
\hline
\multirow{2}{*}{$D_k$} & \multirow{2}{*}{$x^2y+y^{k-1}$} &
$x^2y+y^{k-1}$ $(D_k^+)$ & \multirow{2}{*}{-} &
\multirow{2}{*}{$k \geq 4$} \\ \cline{3-3}
& & $x^2y-y^{k-1}$ $(D_k^-)$ & & \\
\hline
\multirow{2}{*}{$E_6$} & \multirow{2}{*}{$x^3+y^4$} & $x^3+y^4$ $(E_6^+)$ &
\multirow{2}{*}{-} & \multirow{2}{*}{-} \\ \cline{3-3}
& & $x^3-y^4$ $(E_6^-)$ & & \\
\hline
$E_7$ & $x^3+xy^3$ & $x^3+xy^3$ & - & - \\
\hline
$E_8$ & $x^3+y^5$ & $x^3+y^5$ & - & - \\
\hline
\end{tabular}
\end{table}

Using the \textsc{Singular} library \classify{} \citep{classify} for the
complex classification and the one-to-one correspondence between the real main
singularity types and the complex types, the algorithmic classification of a
real germ boils down to determining to which of the corresponding subtypes the
germ is equivalent. For the singularity types $E_7$ and $E_8$, there is nothing
left to do because each of these types has only one real subtype. The rest of
the cases is considered one by one in the following subsections.

Throughout the rest of this article we write $f$ for the given input
polynomial, $g$ for its residual part which can be obtained by applying the
Splitting Lemma, and $c$ for the corank of $f$. We also assume that $f$, and
thus $g$, is a polynomial over $\mathbb Q$. With these notations, $g$ is a
polynomial in $c$ variables.

\subsection{$\boldsymbol{A_1}$}

If $c = 0$, then $f$ is of complex type $A_1$. The residual part in this case
is $g = 0$, even though Table~\ref{tab:normal_forms} assigns the normal form
$x^2$ to this type for formal reasons. As a consequence, all the real
singularities of main type $A_1$ are stably equivalent and their right
equivalence class is completely determined by their inertia index $\lambda$.

\subsection{$\boldsymbol{A_k, k > 1}$}

If $c = 1$, then the singularity is of complex type $A_k$ for some $k > 1$.
Over the real numbers, this type splits up into the subtypes $A_k^+$ and
$A_k^-$ if $k$ is odd. Furthermore $g$ is a univariate polynomial in this case,
say $g\in\mathbb Q[x]$. The value of $k$ is given by the order of $g$ minus
$1$ because $\pm x^{k+1}$ and $g$ are right equivalent and thus have the same
order.

Note that if $k$ is even, then $A_k^+ \requiv A_k^-$ and we have only one real
subtype which we denote by $A_k$. Let $k$ be odd. Then the sign of the
singularity type is determined by the sign of the coefficient of $x^{k+1}$.
This follows since Proposition~\ref{kjet} implies
$\jet(g, k+1) = \pm(\phi_0(x))^{k+1} = \pm(\alpha x)^{k+1}$, where
$\phi(\pm x^{k+1}) = g$, $\alpha \in \R$, and the sign depends on the
singularity type. Since $k+1$ is even and $\alpha \in \mathbb R$, $\phi$ does
not change the sign of the coefficient of $x^{k+1}$. We use
Algorithm~\ref{alg:A_k}, after applying the Splitting Lemma in case $c = 0$ or
$c = 1$.

\begin{algorithm}[ht]
\caption{Algorithm for the case $A_k$}%
\label{alg:A_k}
\begin{algorithmic}[1]

\Require{$f \in \Q[x_1,\ldots,x_n]$ of complex singularity type $A_k$, the
output polynomial $g$ after applying Algorithm~\ref{alg:splitting_lemma}, and
the corank $c$ of $f$}

\Ensure{the real singularity type of $f$, i.e.\@ $A_k$, $A_k^+$ or $A_k^-$,
$k\in\mathbb N$}

\If{$c=0$}
\State type $:=A_1$
\EndIf
\If{$c=1$}
\State $k:= \ord(g)-1$
\If{$k$ is even}
\State type $:=A_k$
\Else
\State $s:=$ coefficient of $x^{k+1}$ in $g$
\If{$s > 0$}
\State type $:=A_k^+$
\Else
\State type $:=A_k^-$
\EndIf
\EndIf
\EndIf
\Return{type}

\end{algorithmic}
\end{algorithm}

\subsection*{}

For the rest of the paper we turn our attention to singularities of corank $2$.
In these cases $0\neq g\in\m^3$ is a polynomial in two variables, say
$g\in\mathbb Q[x,y]$. Using the \textsc{Singular} library {\tt classify.lib},
we determine the complex singularity type and thus the real main singularity
type of $g$, or equivalently $f$. The purpose of the remaining algorithms in
the paper is to determine the correct real subtype of $g$, or equivalently $f$.
We now consider each complex type, or equivalently every real main type,
separately.

\subsection{$\boldsymbol{D_4}$}

The normal form of the complex singularity type $D_4$ is $x^2y+y^3$, which
splits up into $x^2y+y^3$ ($D_4^+$) and $x^2y-y^3$ ($D_4^-$) in the real case.
The two cases can be distinguished by factorization; the details are carried
out in Algorithm~\ref{alg:D_4}. Since the determinacy of $D_4$ is $3$, it
suffices to look at the $3$-jet. The number of factors of the 3-jet over $\R$
is an invariant of the real subtype which is 1 in the case $D_4^+$ and 3 for
$D_4^-$.

However, using the \Singular{} command \verb+factorize+ in order to determine
the number of factors is problematic because the factorization over $\R$
differs from those over $\Q$ and $\C$ in some cases. As an alternative, we
dehomogenize the 3-jet and count the number of real roots of the resulting
univariate polynomial which is exactly the same as the number of factors of the
3-jet over $\R$.

If we want to dehomogenize the 3-jet via $x\mapsto x$, $y\mapsto 1$ without
reducing its degree, we first have to make sure that the coefficient of $x^3$
is non-zero. It is easy to check that this is achieved by lines 2 to 13 of
Algorithm~\ref{alg:D_4}. For the implementation in \Singular{}, we used the
library \texttt{rootsur.lib} \citep{roots} to count the number of real roots
of a univariate polynomial.

\begin{algorithm}[ht]
\caption{\label{alg:D_4}\label{D[4]} Algorithm for the case $D_4$}
\begin{algorithmic}[1]

\Require{$g\in \m^3\subset\mathbb Q[x,y]$ of complex singularity type $D_4$}
\Ensure{the real singularity type of $g$, i.e.\@ $D_4^+$ or $D_4^-$}
\State $h := \jet(g,3)$
\State $s_1:=$ coefficient of ${x^3}$ in $h$
\State $s_2 :=$ coefficient of ${y^3}$ in $h$
\If{$(s_1 = 0)$}
\If{$(s_2 \neq 0)$}
\State swap the variables $x$ and $y$ in $h$
\Else
\State $t_1:=$ coefficient of ${x^2y}$ in $h$
\State $t_2:=$ coefficient of ${xy^2}$ in $h$
\If{$(t_1+t_2 \neq 0)$}
\State apply $x\mapsto x$, $y\mapsto x+y$ to $h$
\Else
\State apply $x\mapsto x$, $y\mapsto 2x+y$ to $h$
\EndIf
\EndIf
\EndIf
\State apply $x\mapsto x$, $y\mapsto 1$ to $h$
\State $n :=$ number of real roots of $h$
\If{$(n<3)$}
\Return $D_4^+$
\Else
\Return $D_4^-$
\EndIf

\end{algorithmic}
\end{algorithm}

\begin{remark}
Geometrically, the dehomogenization in Algorithm~\ref{alg:D_4} corresponds to
blowing the 3-jet up at the origin plus choosing a chart. Since the 3-jet is
homogeneous, blowing-up always yields three lines in the complex case. In the
real case, however, we get either one or three lines depending on their
position w.r.t.\@ the real subspace in the complex picture. All the lines lie
in the chosen chart because the coefficient of $x^3$ is non-zero.
\end{remark}

\subsection{$\boldsymbol{D_k, k > 4}$}

For the cases $D_k$ with $k > 4$, the complex normal form is $x^2y+y^{k-1}$. It
splits up into $x^2y+y^{k-1}$ ($D_k^+$) and $x^2y-y^{k-1}$ ($D_k^-$) for each
$k$ over the reals. We use the following two results from
\citet[p.~35]{Siersma} to distinguish between the two cases:

\begin{lemma}\label{kDeterminacyD[k]k>4}
A singularity of type $D_k^+$ or $D_k^-$ is $(k-1)$-determined.
\end{lemma}

\begin{lemma}\label{transformationD[k]}
Let $j \geq 4$. Then there exists a polynomial
$R \in \m^{j+1} \subset \R[[x, y]]$ such that
\[
x^2y + a_0 x^j + a_1 x^{j-1}y + \ldots + a_j y^j \requiv x^2y + a_j y^j + R,
\quad a_0, \ldots, a_j \in \R,
\]
using the $\R$-algebra automorphism
\begin{align*}
x &\mapsto x + p_1, \text{ where }
p_1 = -\frac{1}{2} (a_1 x^{j-2} + \ldots + a_{j-1} y^{j-2}) \,, \\
y &\mapsto y + p_2, \text{ where }
p_2 = -a_0 x^{j-2} \,.
\end{align*}
\end{lemma}

By Lemma \ref{kDeterminacyD[k]k>4}, the determinacy of a singularity of main
type $D_k$ is $k-1$. Therefore we only need to consider the $(k-1)$-jet of $g$
in this case. By Proposition~\ref{kjet}, the $3$-jet of $g$ factorizes as
$\jet(g, 3) = g_1^2g_2$ over $\R$, where $g_1$ and $g_2$ are homogeneous
polynomials of degree $1$. Note that Lemma~\ref{x^3} ensures that this
factorization can be carried out even over $\Q$. We can thus transform $g$ into
a polynomial of the form
\[
x^2y + \text{terms of degree higher than $3$}
\]
by applying the automorphism defined by $g_1\mapsto x$, $g_2\mapsto y$ to $g$.

We now systematically consider the terms of each degree $3<j<k$. By applying
the transformations in Lemma~\ref{transformationD[k]}, for each $j$, the only
term of total degree $j$ which possibly remains is $a_jy^j$. This term vanishes
for $j<k-1$ and it does not vanish for $j=k-1$, otherwise $g$ is not of complex
type $D_k$. Thus, after applying these transformations, we can write $g$ as
$g=x^2y+\alpha y^{k-1}$ with $\alpha \neq 0$. Clearly if $\alpha>0$ then
$x^2y+\alpha y^{k-1}\requiv x^2y+y^{k-1}$ and if $\alpha<0$ then
$x^2y+\alpha y^{k-1}\requiv x^2y-y^{k-1}$.

\begin{algorithm}[ht]
\caption{\label{alg:D_k} Algorithm for the case $D_k$, $k > 4$}
\begin{algorithmic}[1]

\Require{$g \in \m^3\subset\mathbb Q[x,y]$ of complex singularity type $D_k$,
$k\in\mathbb N$, $k>4$}

\Ensure{the real singularity type of $g$, i.e.\@ $D_k^+$ or $D_k^-$}

\State $k:= \mu(g)$
\State $h:=\jet(g,k-1)$
\State factorize $\jet(h,3)$ as $h_1^2h_2$, where $h_1$ and $h_2$ are linear
\State apply $h_1\mapsto x$, $h_2\mapsto y$ to $h$
\For{$(j = 4, \ldots, k-1)$}
\If{$(\jet(h,j)-x^2y\neq0)$}
\State write $\jet(h,j)-x^2y$ as
$a_0x^j+a_1x^{j-1}y+\cdots +a_jy^j,\quad a_0,\ldots a_j\in\mathbb Q$
\State apply $x\mapsto x-\frac{1}{2}(a_1x^{j-2}+\cdots
+a_{j-1}y^{j-2})$, $y\mapsto y-a_0x^{j-2}$ to $h$
\State $h:=\jet(h,k-1)$
\EndIf
\EndFor
\State write $h$ as $h=x^2y+\alpha y^{k-1}$, $0\neq\alpha\in\mathbb Q$
\If{$(\alpha>0)$}
\Return $D_k^+$
\Else
\Return $D_k^-$
\EndIf

\end{algorithmic}
\end{algorithm}

\subsection{$\boldsymbol{E_6}$}

In this case, whose complex normal form is $x^3+y^4$, we have that either
$g \requiv x^3+y^4$ ($E_6^+$) or $g \requiv x^3-y^4$ ($E_6^-$). Therefore there
exists an $\R$-algebra automorphism $\phi$ of $\R[[x,y]]$ such that
$g = (\phi(x))^3+(\phi(y))^4$ or such that $g = (\phi(x))^3-(\phi(y))^4$. Since
the coefficients of $x^3$ and $y^3$ in $g$ cannot both be zero, we can ensure
that the coefficient of $x^3$ is non-zero by swapping the variables if
necessary. Now, by Proposition~\ref{kjet} and Lemma~\ref{x^3}, the $3$-jet of
$g$ factorizes as $c(g_1)^3$ with $c \in \Q$ and $g_1 = b_0x+b_1y \in \Q[x,y]$,
$b_0 \neq 0$. By applying $x \mapsto \frac{x-b_1y}{b_0}$, $y \mapsto y$ to $g$,
we can thus assume without loss of generality that $\phi_0$ is of the form
$\phi_0(x) = c'x$, $\phi_0(y) = d_0 x + d_1 y$ with $c', d_0, d_1 \in \R$.
Since $\phi$ is an automorphism, we have that $d_1 \neq 0$. Hence
\[
(\phi(y))^4 = d_1^4 y^4
+ (\text{terms of degree 4 and higher, not of the form $\alpha y^4$,
$\alpha \in \R$}) \,.
\]
If we can show that $(\phi(x))^3$ does not contain a term of the form
$\alpha y^4$, $\alpha \in \R$, then we can determine whether $g$ is of type
$E_6^-$ or $E_6^+$ by considering the sign of the coefficient of the monomial
$y^4$. A simple calculation yields
\begin{align*}
\jet((\phi(x))^3, 4) - \jet((\phi(x))^3, 3)
&= 3 (\phi_0(x)^2) (\phi_1(x)-\phi_0(x)) \\
&= 3 (c'x)^2 (\phi_1(x)-\phi_0(x)) \,,
\end{align*}
which means that $(\phi(x))^3$ does not have any term of the form $\alpha y^4$,
$\alpha \in \R$.

\begin{algorithm}[ht]
\caption{\label{alg:E_6}\label{E[6]} Algorithm for the case $E_6$}
\begin{algorithmic}[1]

\Require{$g\in \m^3\subset\mathbb Q[x,y]$ of complex singularity type $E_6$}
\Ensure{the real singularity type of $g$, i.e.\@ $E_6^+$ or $E_6^-$}
\State $h:= \jet(g,3)$
\State $s:=$ coefficient of ${x^3}$ in $h$
\If{$(s=0)$}
\State swap the variables $x$ and $y$
\EndIf
\State factorize $h$ into linear factors over $\mathbb Q[x,y]$, with a factor
$g_1=b_0x+b_1y$
\State apply $x\mapsto \frac{x-b_1y}{b_0}$, $y\mapsto y$ to $g$
\State $d :=$ coefficient of $y^4$ in $g$
\If{$(d>0)$}
\Return $E_6^+$
\Else
\Return $E_6^-$
\EndIf

\end{algorithmic}
\end{algorithm}

\section{Acknowledgements}

We would like to thank Gerhard Pfister and Gert-Martin Greuel for insightful
conversations. We also would like to thank the two anonymous referees whose
suggestions greatly improved this article.

\end{document}